\documentclass[a4paper]{article}
\usepackage[affil-it]{authblk}
\usepackage{amsmath,amssymb,latexsym}
\usepackage{amsfonts,amstext,amsthm,amscd}
\usepackage{graphicx}
\usepackage{enumerate}
\usepackage{amsthm}
\usepackage{xypic}
\usepackage{color}
\theoremstyle{plain}
\theoremstyle{definition}
\newtheorem{theorem}{Theorem}[section]
\newtheorem{corollary}[theorem]{Corollary}
\newtheorem{definition}[theorem]{Definition}

\newtheorem{lemma}[theorem]{Lemma}

\newtheorem{remark}[theorem]{Remark}
\numberwithin{equation}{section}

\newcommand{\UC}{\mathbb{S}^{1}}

\newcommand{\orb}{\mathcal{O}}

\newcommand{\Homeo}{\mathrm{Homeo}}
\newcommand{\HHomeo}{\widetilde{\mathrm{Homeo}}}

\newcommand{\Ll}{\mathcal{L}}

\newcommand{\Q}{\mathbb{Q}}
\newcommand{\R}{\mathbb{R}}

\newcommand{\Zz}{\widehat{\mathbb{Z}}}
\newcommand{\Ss}{\mathcal{S}}

\newcommand{\Z}{\mathbb{Z}}

\newcommand{\Tn}{\mathbb{T}^{2}}
\newcommand{\T}{\mathbb{T}}

\newcommand{\To}{\longrightarrow}
\newcommand{\goTo}{\longmapsto}
\newcommand{\id}{\mathrm{id}}

\newcommand{\Cont}{\mathrm{C}}

\newcommand{\homeo}{\mathrm{Homeo}}

\newcommand{\HomeoL}{\widetilde{\mathrm{Homeo}_+}(\Ss)}
\newcommand{\HomeoS}{\mathrm{Homeo_+}(\Ss)}
\newcommand{\lp}{\mathrm{lp}}

\newcommand{\Hull}{\Omega}

\begin{document}%

\title{Dynamics of induced homeomorphisms of one-dimensional solenoids}
\author{ Francisco J. L\'opez--Hern\'andez}
\affil{Centro de Investigaci\'on en matem\'aticas A.C.}
\date{\today}
\maketitle
\begin{abstract}
We study the displacement function of homeomorphisms isotopic to the identity of the universal one-dimensional solenoid and we get a characterization of the lifting property for an  open and dense subgroup of the isotopy component of the identity. The dynamics of an element in this subgroup is also described using rotation theory.

\end{abstract}

%\tableofcontents
\section*{Introduction}
H. Poincar\'e (see \cite{Poi})   introduced an invariant of topological
conjugation called the \emph{rotation number} for orientation-preserving  homeomorphisms of the unit circle, defined through the lifting property to the universal covering space $\R$ of the circle. \\

Denote by $\Homeo_{+}(\UC)$ the group of homeomorphisms which are orientation-preserving of $\UC$ and let $\HHomeo_+(\UC)$ be the space containing the real valued functions that are lifts of elements in $\Homeo_{+}(\UC).$ Define $\tau:\HHomeo_+(\UC)\To \R$ by
$$\tau(F)= \lim_{n\To\infty}\frac{F^{n}(x)-x}{n}.$$
Since the function $\tau$ is $\Z$--invariant and it does not depend on the choice of $x$,  it can be projected to $\rho(f)= \pi(\tau(f)) \in\UC.$  Therefore, for any $f\in\Homeo_{+}(\UC)$ we have a distinguished element $\rho(f)$  called the \emph{Poincar\'e's rotation number,} which gives information on the topological dynamics generated by $f$, summarized in the following two cases:
 \begin{enumerate}
\item $\rho(f)$ is rational if and only if $f$ has a periodic orbit;
\item if $\rho(f)$ is irrational, then  $f$ is semi-conjugate to the irrational rotation by $\rho(f).$
\end{enumerate}

\smallskip

 This theory was generalized to \emph{rotation sets} for  toral homeomorphisms isotopic to the identity in a similar way as in $\Homeo_{+}(\UC)$. The rotation set of any $f\in\Homeo_{+}(\Tn)$, also denoted by $\rho(f)$, is a compact and convex subset of $\R^{2}$ (see \cite{M-Z}).\\

 On the one hand, Frank's theorem gives dynamical information  (see \cite{Fra}) similar to part (1) of above   when $f\in\Homeo_{+}(\Tn)$ such that $\rho(f)$ has nonempty interior. If a vector $v$ lies in the interior of $\rho(f)$ and has both coordinates rational, then there is a periodic point $x\in\Tn$ with the property that
$$\frac{F^{q}(x_{0})-x_{0}}{q}=v,$$
where $x_{0}\in\R^{2}$ is any lift of $x$ and $q$ is the least period of $x$.\\

On the other hand, T. Jaeger obtained a version of the semi-conjugation theorem (see \cite{Jag}) as in part (2) of Poincar\'e 's theorem for irrational pseudo-rotations with bounded mean motion.\\

In the case of homeomorphisms from the real line with bounded displacement, J. Kwapisz in \cite{Kwa} gives a nice generalization of the rotation set presented in the torus case, particulary when the displacement function is an almost periodic function. Some aspects of the rotation theory for arbitrary manifolds were introduced by M. Pollicott in \cite{Pol} and recently this theory was generalized for solenoidal groups by A. Verjovsky and M. Cruz L\'opez in \cite {C-V},  giving a semi-conjugation theorem for irrational pseudo-rotations with bounded mean motion. A detailed summary on other generalizations of this theory can be found in \cite{A-J}.\\

Our goal in this article is to describe the dynamics of a certain class of homeomorphisms of the universal one--dimensional solenoid which can be described as follows. \\

The \emph{universal one--dimensional solenoid} is the inverse limit of the unbrached covering tower of $\UC$
 $$\Ss:=\lim_{\overleftarrow{n}} (\UC,p_{n}),$$ together with surjective homomorphisms  determined by projection onto the $n^{\text{th}}$ coordinate $$\pi_{n}:\Ss\To \UC.$$ The first projection $\pi_{1}$ defines a locally
 trivial $\Zz$--bundle structure $$\Zz\hookrightarrow \Ss \To \UC,$$ where
 $$\Zz:=\lim_{\overleftarrow{n}}\; \Z/n\Z$$ is the profinite completion of $\Z$, which
 is a compact perfect totally disconnected abelian topological group
 homeomorphic to the Cantor set. Since $\Z$  is residually finite, its profinite completion $\Zz$ admits a dense inclusion of $\Z$. There is also defined a dense canonical inclusion of $\R$,   $\sigma:\R\To\Ss$.\\

In Section \ref{displacements} of the article the displacement function for orientation-preserving   solenoidal homeomorphims isotopic to the identity  will be analyzed, and we will  compare this function with the displacement function for orientation-preserving  circle homeomorphisms. If $f\in\Homeo_{+}(\UC)$   we can lift it to some homeomorphism $F:\R\To\R$, and any lift satisfies that the displacement function 
 $$\delta:=F-\id:\R\To\R$$
 is periodic. The displacement function for   orientation-preserving  solenoidal homeomorphisms isotopic to the identity,  $\Homeo_{+}(\Ss)$, can be defined using their lifts to the appropriate covering space $\R\times\Zz$ (see Section \ref{solenoid}). These lifs can be described as $$(x,k)\goTo (F_{k}(x),k):=(x+\delta_{k}(x),k),$$
where $F_{k}:\R\To\R$ is an increasing continuous function and $\delta_{k}:\R\To\R$ is a bounded continuous function. Therefore for all $k\in\Zz$,  the \emph{displacement function} is defined as $$\delta_{k}=F_{k}-\id:\R\To\R.$$

Recall now the $\R$-action in $C(\R)$ given by $$(s,\delta)\goTo\delta^{s}$$
where $\delta^{s}:\R\To\R$ is defined by $\delta^{s}(x)=\delta(x+s)$, and denote by $\Hull(\delta)$ the closure of the orbit of  $\delta$. The function  $\delta$  is called  \emph{almost periodic} if $\Hull(\delta)$ is compact. In this case it is possible to  define on it a group structure. The classification theorem for almost periodic functions (see \cite{Bohr}) states that $\Hull(\delta)$ is isomorphic to the circle in the \emph{periodic} case, $\T^{n}$ with $n>1$ in the \emph{quasi-periodic} case or a solenoidal group in the \emph{purely limit periodic} case.\\

If $\delta_{0}$ is the displacement function of $f\in\Homeo_{+}(\Ss)$ for $k=0$, it will be proved that $\Omega(\delta_0)$ is a quotient topological group of $\Ss$ of the form  $$\Hull(\delta_{0})\cong \Ss/\ker(K),$$
where $\ker(K)$ is the kernel of a certain continuous homomorphism $K:\Ss\To\Hull(\delta_{0})$. An immediate consequence of this fact is the next description of the displacement function for  orientation-preserving solenoidal homeomorphims isotopic to identity.\\

 \textbf{Theorem \ref{lp}}
 For all $k\in\Zz$ the  displacement function $\delta_{k}:\R\To\R$ is periodic or purely limit periodic.\\

In Section \ref{semi-conjugation} is proved that given $f\in\Homeo_{+}(\Ss)$ with displacement function for $k=0$ denoted by $\delta_{0}$, there exists an  orientation-preserving  homeomorphism istopic to the identity $g:\Hull(\delta_{0})\To \Hull(\delta_{0})$ such that $g$ is  semi-conjugated to $f$ by $K$, i.e. the following diagram commutes
 	\[\xymatrix{\Ss \ar[r]^{f}\ar[d]^{K}& \Ss\ar[d]^{K} \\
 		\Hull(\delta_{0})\ar[r]^{g}& \Hull(\delta_{0})}\]

An example on how this semi-conjugation gives us  information on the dynamics generated by $f$ is when $$\Hull(\delta_{0})\cong \Ss/\ker(K)\cong \UC$$ called the subspace of \emph{induced homeomorphisms} by a periodic function and denoted by $\Homeo_{I}(\Ss)$ (compare with \cite{R-T-L}). This group is studied in Section \ref{Inducedhomeos} for the case of orientation-preserving homeomorphisms isotopic to the identity. \\

In the first part of the section,  a description of such homeomorphisms is given, in which  for some $n\in\Z_{+}$, we have an orientation-preserving  homeomorphism $f_{n}:\R/n\Z\To\R/n\Z$ such that the following  diagram commutes:
 	\[\xymatrix{\Ss \ar[r]^{f}\ar[d]^{\pi_{n}}& \Ss\ar[d]^{\pi_{n}} \\
 		\R/n\Z \ar[r]^{f_{n}}& \R/n\Z}\]

 Also we will give a description of the lift of these kinds of homeomorphisms, which will be denoted by $\Homeo_{I_{n}}(\Ss)$ for each $n\in\Z_{+}$, and will be called  \emph{induced homeomorphism of degree $n$} by the homeomorphims $f_{n}\in\Homeo_{+}(\R/n\Z)$. A first observation is that if $n|m$, where $n,m\in\Z$, then $$\Homeo_{I_{n}}(\Ss)\subset\Homeo_{I_{m}}(\Ss),$$ and we will have that
   $$\Homeo_{I}(\Ss)=\lim_{\overrightarrow{n}}\Homeo_{I_{n}}(\Ss)=\bigcup_{n\in\Z_{+}}\Homeo_{I_{n}}(\Ss)$$

 This subspace has a group structure and coincides with the dense subspace of $\Homeo_{+}(\Ss)$ which has periodic displacement. Also we will prove that $$\Homeo_{I_{1}}(\Ss)\cong\HHomeo_{+}(\UC),$$
fitting  the  universal central extension (see \cite{Ghys})
$$0\To\Z\To\HHomeo_{+}(\UC)\To\Homeo_{+}(\UC)\To 1$$
into the diagram
\[\xymatrix{0\ar[r]\ar[d]&\Z\ar[r]\ar[d]& \HHomeo_{+}(\UC)\ar[r]\ar[d]&\Homeo_{+}(\UC)\ar[r]\ar[d]& 1\ar[d]\\
	0\ar[r]&R_{\alpha}\ar[r]& \Homeo_{I_{1}}(\Ss)\ar[r]&\Homeo_{+}(\UC)\ar[r]& 1}\]
where $R_{\alpha}$ denotes integer translations in $\Ss$.\\

Finally, the dynamics of  induced homeomorphisms can be described using the Poincar\'e theory for $\homeo_{+}(\UC)$. In \cite{C-V}, for general homeomorphisms which are isotopic to the identity, the authors study the irrational case and do not consider rational dynamics.  In order to complete the dynamical picture of Poincar\'e theory in the case of induced homeomorphisms, we introduce first a convenient concept which captures the idea of rationality.\\%In \cite{Ali} the author studies rational translations but not general aspects of a rational dynamics.

\textbf{Definition \ref{fiberperiodic}}
	Let $f\in\Homeo_{+}(\Ss)$. We will say that $s\in\Ss$ is   \emph{$p/q$-fiber periodic} if there  are $p,q\in\Z_+$ such that $$f^{q}(s)=s+\sigma(p).$$
	 Here the sum is the Abelian sum along the leaves. Note that the name ``fiber periodic"  is due to the fact that the point is returning to the $p$-fiber in $s$ periodically after $q$ times, and we refer to the orbit of $s$ as a \emph{$p/q$-fiber}.
\\

The relationship between the dynamics generated by the homeomorphism  $f_{1}\in\Homeo_{+}(\UC)$ and the dynamics generated by the induced homeomorphism $f\in\Homeo_{I}(\Ss)$
is described in next.\\

\textbf{Theorem \ref{dynamics}}
Let $f\in\Homeo_{+}(\Ss)$ be induced by a homeomorphism $f_{1}\in\Homeo_{+}(\UC).$
\begin{enumerate}
\item If $\rho(f)=p/q,$ then any point $s\in\Ss$ is a $p/q$-fiber periodic point or the orbit of $s$ is asymptotic to the orbit of a $p/q$-fiber periodic point.
 \item If $\rho(f)\notin\Q$, then $f$ is semi-conjugate to the rotation by  $\rho(f).$
\end{enumerate}

\smallskip
Section \ref{solenoid} introduces the universal one-dimensional solenoid and the lifting properties of its homeomorphisms. In Section \ref{displacements} we study the displacement function for  orientation-preserving solenoidal homeomorphisms isotopic to the identity. Section \ref{semi-conjugation} deals with the semi-conjugation to a quotient dynamics, and finally, Section \ref{Inducedhomeos} talks about the dynamics generated by induced homeomorphisms.

\section{The solenoid and its homeomorphisms that are isotopic to the identity}\label{solenoid}
In this section, the universal one-dimensional solenoid will be introduced which is the space where we are interested in studying the dynamics generated by orientation-preserving homeomorphisms isotopic to the identity. Also these kinds of homeomorphims and their lifting properties will be studied  the end of this section.\\ 
    \subsection{The solenoid}
 For every integer $n\geq 1$ we have defined the unbranched covering space of degree $n$  by $p_n:\UC \To \UC$, $z\longmapsto {z^n}$.  If $n,m\in \Z_{+}$ and $n$ divides $m$, then
 there exists a unique covering map $p_{nm}:\UC\To \UC$ such that
 $p_n \circ p_{nm} = p_m$. This determines a projective system of covering spaces $\{\UC,p_n\}_{n\geq 1}$
 whose projective limit is the \emph{universal one--dimensional solenoid}
 $$\Ss:=\lim_{\overleftarrow{n}} (\UC,p_{n}).$$ We have canonical projections  determined by projection onto the $n^{\text{th}}$ coordinate $$\pi_{n}:\Ss\To \UC.$$ This determines a locally
 trivial $\Zz$--bundle structure $\Zz\hookrightarrow \Ss \To \UC$, where
 $$\Zz:=\lim_{\overleftarrow{n}}\; \Z/m\Z$$ is \emph{the profinite completion of $\Z$}, which
 is a compact perfect totally disconnected abelian topological group
 homeomorphic to the Cantor set. Since $\Zz$ is the  profinite completion of $\Z$, we have canonical projections  determined by projection onto the $n^{\text{th}}$ coordinate $$p_{n}:\Zz\To \Z/n\Z,$$
and  $\Zz$ admits a canonical inclusion $i:\Z\To\Zz$ defined by $t\goTo(p_{n}(t))_{n\in\Z_{+}},$ whose image is dense. We will use $t$ to denote $i(t)\in\Zz$.\\

 We can define a properly discontinuously free action of $\Z$ on $\R\times \Zz$  by
 $$t\cdot (x,k) := (x+t,k-t) \quad (t\in \Z).$$
Here $\Ss$ is identified with the orbit space
 $\R\times_{\Z} \Zz \equiv \R\times \Zz / \Z$, and  $\Z$ is acting on $\R$ by
 covering transformations and on $\Zz$ by translations. The path--connected component $\Ll_0$ of the identity element
 $0\in \Ss$ is called the \emph{base leaf}. Clearly, $\Ll_0$ is the image of
 $\R\times \{0\}$ under the canonical projection $\R\times \Zz\To \Ss$ and it is homeomorphic to $\R$.
 \\

 In summary, $\Ss$ is a compact connected abelian topological group and also is a
 one--dimensional lamination where each ``leaf" is a simply connected one--dimensional manifold homeomorphic to the universal covering space $\R$ of $\UC$, and a typical 'transversal section' is isomorphic to the Cantor group $\Zz$.\\

 %Recall that $\Ss$ is the orbit space of $\R\times \Zz$ of the $\Z$--action
% $$\gamma\cdot (x,t) = (x+\gamma,t-\gamma)\qquad (\gamma\in \Z).$$
\subsection{ Homeomorphisms that are isotopic to the identity}
 Let $p:\R\times \Zz\To \Ss$ denote the canonical projection. Then $p$
 is an infinite cyclic covering and we have a lifting property of homeomorphisms. The space of  orientation-preserving homeomorphisms isotopic to the identity   of $\Ss$ is denoted by $\Homeo_{+}(\Ss),$ and the space containing all the liftings of homeomorphisms in $\Homeo_{+}(\Ss)$ will be denoted by $\HHomeo_{+}(\Ss)$. By \cite{Kwa2} we have a complete description of the homeomorphisms in $\HHomeo_{+}(\Ss)$.\\

 Let $F:\R\times \Zz\To \R\times \Zz$ be a lifting of $f\in \HomeoS$ to
 $\R\times \Zz$. Then $F$ has the form
 $$F(x,k) = (F_k (x),k),$$
where $F_{k}:\R\To\R$ satisfies the condition of being equivariant with respect to the $\Z$--action:
 $$F_{k-t}(x+t) = F_k (x) + t$$
 for any $t\in \Z$. We have  a continuous function given by $\Zz\To \homeo_{+}(\R)$
 $k\longmapsto F_k$, where $F_t:\R\To \R,$
 and $R_t:\Zz \To \Zz$ is a minimal translation. This implies that $F$  commutes with the integral
 translation $T_t:\R\times \Zz\To \R\times \Zz$ given by
 $$(x,k)\longmapsto (x+t,k-t)$$ and also must be invariant under the
 $\Z$--action in $\Cont(\Zz,\homeo(\R))$.\\

Moreover, $F_k:=\id+\delta_{k}$ where $\delta_{k}:\R\To\R$ is an increasing, bounded and continuous function. The function $\Delta:\R\times\Zz\To\R$ defined by $(x,k)\goTo \delta_{k}(x)$ satisfies $$T_{t}\circ\Delta(x,k)=\Delta(x,k),$$ i.e is invariant by integral translation and induces a continuous functions $\delta:\Ss\to\R$ such that $$f:\id+i\circ\delta.$$

 Denote by $\HomeoL$ the set of all such homeomorphisms. The $\R-$displacement of the homeomorphism in $\HomeoL$ will be described in the following section.

\section{Limit periodic displacements}\label{displacements}
% %
 This study began in the work  \cite{Lop} where  the displacement function  is introduced and studied as in this section.\\
 
 If $F\in\HHomeo_{+}(\Ss)$, then the \emph{displacement function}  $D_{F}:\widehat{\Z}\To C(\R)$ can be defined  as
 $$k\goTo F_{k}-\id=\delta_k$$ For all $k\in\Zz$   
 
 \begin{lemma}
  If $F\in\HHomeo_{+}(\Ss)$,  the displacement function
 is continuous and closed.
 \end{lemma}
 \begin{proof}
  Since $\delta:\R\times\Zz\To\R$ is continuous, by the exponential law $A_{F}$ is continuous. Applying the closed function theorem we conclude that $A_{F}$ is closed.
\end{proof}
% %
 Remember that if $C(\R)$ denote the set of all continuous functions from $\R$ to $\R$ with the compact-open topology, then  an $\R-$action on  $C(\R)$  is defined by  $$\R\times C(\R)\To C(\R),\quad (\varphi,t)\goTo \varphi^{t},$$
 where $\varphi^{t}:\R\To\R$ is given by
 $$\varphi^{t}(x):=\varphi(x+t).$$
 Denote by $\orb_{\R}(\varphi)$  the orbit of $\varphi$ under this action,  and by $\Hull(\varphi)$  the  closure of $\orb_{\R}(\varphi)$ in $C(\R)$. We will say that $\varphi$  is \emph{almost periodic} if $\Hull(\varphi)$ is compact. In this case we can define a group structure `` $*$''  so that  for  %on $\Hull(\varphi)$ by $\varphi^{t}*\varphi^{s}$ and  
 $\lim_{n\To\infty}\varphi^{t_{n}}*\lim_{n\To\infty}\varphi^{s_{n}}=\lim_{n\To\infty}\varphi^{t_{n}+s_{n}}.$ Note that %is not the usual sum of continuous functions and% 
 $(\Hull(\varphi),*)$ is a compact abelian  topological group with neutral element $\varphi$.\\

% %
% %
  If $\alpha:\Ss\To\R$ is a continuous function, define
  $$K_{0}:\Ll_0\subset\Ss\To C(\R),\;0 \goTo \alpha_{0},$$
  where $$\alpha_{0}:\R\To\R$$
  is defined by
  $$\alpha_0:=\alpha\circ\sigma,$$
 and $\sigma:\R\hookrightarrow\Ss$ is the   1-parameter dense subgroup. By definition
  $$\orb(\alpha_{0}):=\{\alpha_{0}^{t}\in C(\R):\alpha_{0}^{t}(x)=\alpha_{0}^{t}(x+t)\}.$$
 Then
 $$K_{0}(\sigma(t))=\alpha_{0}^{t}\quad(t\in\R).$$
 %
% %
  \begin{theorem}
 For any continuous function $\alpha:\Ss\To\R$, the function $K_{0}$ can be extended to a  continuous   and surjective homomorphism  $K:\Ss\To\Hull(\alpha_0)$. Therefore $\Hull(\alpha_{0})$ is a quotient group from $\Ss$.
 \end{theorem}
 \begin{proof}
 First we will prove that $K_{0}:\Ll_0\To\Hull(\alpha_{0})$ defines a  surjective homomorphism. If $\varphi\in\orb(\alpha_{0}),$ then $\varphi(t)=\alpha_{0}(t+r)$
 for some $r\in\R$. By definition
 \begin{eqnarray*}
 \alpha_{0}(t+r)&=&\alpha(\sigma(t+r))\\
                 &=&\alpha(\sigma(t)+\sigma(r))\\
                 &=&\alpha_{\sigma(r)}(t),
 \end{eqnarray*}
  for all $t,r\in\R.$ This tell us that $K_{0}$ is surjective. Also, the products in $\Ll_{0}$ and $\orb(\alpha_{0})$ are given by the additive structure in $\R,$ which means $K_{0}$ is a continuous homomorphism. Note that this homomorphism can be naturally extended to a surjective homomorphism
 $$K:\Ss\To \Hull(\alpha_{0}),$$
  Since $\Ss$ is compact and $\Ll_{0}$ is dense, the traslation flow is an isometry, $\alpha$ is uniformly continous and $K_{0}$ is continuous. The  image of   $K_{0}$ under this extension  to $\Ss$ is closed in $\Hull(\alpha_{0})$ and must contain  $\orb(\alpha_{0}).$ Since $\orb(\alpha_{0})$ is dense in $\Hull(\alpha_{0})$, it follows  that $\mathrm{Im}(K)=\Hull(\alpha_{0}).$ Applying the first isomorphism  theorem, we conclude that
% %
 $$\Hull(\alpha_{0})\cong\Ss/\mathrm{ker}(K).$$
 \end{proof}
 \begin{remark}
  Since $\Hull(\alpha_{0})$ is a  quotient group  of $\Ss,$ it follows that $\Hull(\alpha_{0})$ is compact and $\alpha_{0}$ is almost-periodic. If $\alpha_{0}$ is quasi-periodic, then
   $$\T^{n}\cong\Hull(\alpha_{0})\cong\Ss/\mathrm{ker}(K)$$
    for some $n\geq 2.$ Given that $\T^{n}$ cannot be a quotient group  of $\Ss$ for $n\geq 2,$ it follows that $\delta'_{0}$ is not quasi-periodic, so must be  periodic or purely limit periodic.
 \end{remark}
 \begin{theorem}\label{lp}
 For all $k\in\Zz$ the  displacement function $\delta_{k}:\R\To\R$ is periodic or purely limit periodic.
 \end{theorem}
 \begin{proof}
  We know that for all $k\in\Zz$ and arbitrary $n\in\Z$, $$\delta_{k+n}(x)=\delta_{k}(x+n).$$
  This implies That $\delta_{k}$ and $\delta_{k+n}$ are in the same orbit under the $\R-$action  in $C(\R)$. This means that if $\delta_{k}$ is limit periodic for some $k\in\Zz,$ then $\delta_{k+n}$ will be limit periodic for every $n\in\Z.$ Thus, it is enough  to prove that for some fixed $k\in\Zz$,  $\delta_{k}$ is limit periodic. We will prove this for $\delta_{0}.$\\

 If $\delta:\R\times\Zz\To\R$ is a bounded continuous function and $\Z$-invariant, then $\delta$  induces a continuous function $\overline{\delta}:\Ss\To\R$.
 Denote by  $\overline{\delta_{0}}$  the function obtained  in the proof of the last theorem. It is easy to see that $\overline{\delta_{0}}$ coincides with the function $\delta_{0}$ since we have the following commutative diagram
% %%
% %%
 \[\xymatrix{\R \ar@/^0.7cm/[rr]^{\delta_{0}}\ar[r]^{\sigma'} \ar[dr]^{\sigma} & \R\times\Zz\ar[d]^{P}\ar[r]^{\delta}& \R\\
   & \Ss\ar[ur]_{\overline{\delta}}&}\]
%@{.>}
 Here $\sigma'(x)=(x,0),$ therefore
% %
 \begin{eqnarray*}
 \delta_{0}&=&\delta\circ \sigma'\\
 &=&\overline{\delta}\circ P\circ \sigma'\\
 &=& \overline{\delta}\circ \sigma\\
 &=&\overline{\delta}_{0}.
 \end{eqnarray*}
% %
% %
% %
  Since $\overline{\delta_{0}}$ is limit periodic, it follows that  $\delta_{0}$ is limit periodic. So then $\delta_{k}$ is limit periodic for every $k\in\Zz$.
% %
 \end{proof}
% \begin{example}
% Define $\overline{\delta}:\Ss\To \R$ as $\overline{\delta}(s)=d(s,0)$ la distance function to 0. It's easy to see that in the last construction that  $$\mathrm{Ker}(K)=\{0\},$$
%     by properties metric. So that $\Hull(\alpha)\cong \Ss$. If we define the homeomorphism determineded by the displacement  $\overline{\delta},$ $f:\Ss\To\Ss$ given by $$f(s)=s+\alpha(s),$$ we obtain a homeomorphism  isotopic to the  identity whit limit periodic displacement.
% %
% %
% %%\end{enumerate}
%\end{example}
% %
Denote by $\Homeo_{+}^{p}(\Ss)\subset\Homeo_{+}(\Ss)$  the homeomorphisms having periodic displacement function and by $\Homeo_{+}^{lp}(\Ss)\subset\Homeo_{+}(\Ss)$  the homeomorphisms having purely limit periodic displacement function. The obvious theorem is the following.

\begin{theorem}
$\Homeo_{+}(\Ss)=\Homeo_{+}^{p}(\Ss)\cup\Homeo_{+}^{lp}(\Ss)$. Moreover, $\Homeo_{+}^{p}(\Ss)$ is a dense subgroup.
\end{theorem}
\begin{proof}
$\Homeo_{+}^{lp}(\Ss)$ is closed since it contains its limit points. Moreover each function in  $\Homeo_{+}^{lp}(\Ss)$ can be aproximated by functions in $\Homeo_{+}^{p}(\Ss)$, and therefore $\Homeo_{+}^{p}(\Ss)$ is dense.
\end{proof}
 \begin{theorem}\label{Ap}
  Let  $F\in\HHomeo_{+}(\Ss)$ be given by  $F(x,k)=(x+\delta_{k}(x),k).$ If $D_{F}:\Zz\to C(\R)$ is injective, then  $D_{F}(k):=\delta_{k}$ is limit periodic for all $k\in\Zz.$
 \end{theorem}
 \begin{proof}
 If $D_{F}$ is injective then $\delta_{k}\neq \delta_{k+n}$ for all $n\in\Z.$ By the $\Z$-invariance of $\delta$, we know that for every $x\in\R$
 $$\delta_{k}(x)\neq \delta_{k+n}(x)=\delta_{k}(x+n).$$
 We conclude that  $\delta_{k}$ cannot be periodic, so must be limit periodic.
 \end{proof}

We will give a charactization of $\Homeo_{+}^{p}(\Ss)$ in  Section \ref{Inducedhomeos}.

\section{The semi-conjugation theorem}\label{semi-conjugation}
In this seccion  the dynamics generated by $f\in\Homeo_{+}(\Ss)$ will be compared with the dynamics generated by a homeomorphism of a qoutient group of $\Ss$.\\

 Given  $f\in\homeo_{+}(\Ss)$, the displacement function at the level 0, $\delta_{0}$, satisfies $$\Hull(\delta_{0})\cong \Ss/\ker(K) ,$$
where $\ker(K)$ is the kernel of a specific homomorphism $K$. Now we would like to give a homeomorphism isotopic to the identity   $g:\Hull(\delta_{0})\To \Hull(\delta_{0})$ such that is semi-conjugated to $f$ by $K$ i.e. the following diagram commutes
 	\[\xymatrix{\Ss \ar[r]^{f}\ar[d]^{K}& \Ss\ar[d]^{K} \\
 		\Hull(\delta_{0})\ar[r]^{g}& \Hull(\delta_{0})}\]
%Remember that a tipical element $\gamma$ in $\Hull(\delta)$ is a function $\delta_{s}:\R\To\R$ such that $\delta_{s}(x)=\delta(x+s)$ if $\gamma$ is the orbit of $\gamma$ by the $\R$-action or is a function that is uniform limit of this kind of functions.\\

First, given a limit periodic function $\delta:\R\To\R$,  suppose that $F:\R\To\R$ defined as $x\goTo x+\delta(x)$ is an increasing homeomorphism, define   $g:\Hull(\delta)\To\Hull(\delta)$  by $$\gamma\goTo\gamma*\delta^{\gamma(0)}$$
where ``$*$'' denotes the product defined on $\Hull(\delta)$ in the section \ref{displacements}.
\begin{lemma}
$g$ defines a homeomorphism isotopic to the identity.
\end{lemma}
\begin{proof}
We notice that $g$ is a continuous function because $*$ and the ``valuation function'' are continuous. Since $\Hull(\delta)$ is compact, it is enough to prove that g is a bijective. \\

This function is onto since it is onto on $\orb_{\R}(\delta)$ which is a dense set in $\Hull(\delta)$ and $g$ is continuous.\\

To prove that $g$ is one to one, we can take $\alpha,\;\beta\in\Hull(\delta)$ such that $\alpha=\delta^{a}$  and $\beta=\delta^{b}$ and suppose that $g(\beta)=g(\alpha)$ then
 $$\alpha*\delta^{\alpha(0)}=\beta*\beta^{\beta(0)}.$$
 By definition $\alpha*\delta^{\alpha(0)}=\delta^{a+\delta(a)}$ and $\beta*\delta^{\beta(0)}=\delta^{b+\delta(b)}$. Therefore
 $$\delta(x+a+\delta(a))=\delta(x+b+\delta(b))$$
For all $x\in\R$. Using $x=-\delta(a)$ it satisfies
$$\delta(a)=\delta(b+\delta(b)-\delta(a)).$$ 

But 

\begin{eqnarray*}
F(b+\delta(b)-\delta(a))&=&b+\delta(b)-\delta(a)+\delta(b+\delta(b)-\delta(a))\\
                        &=&b+\delta(b)-\delta(a)+\delta(a)\\
                        &=&b+\delta(b)\\
                        &=&F(b).
                       \end{eqnarray*}
Since $F$ is  an increasing homeomorphism $b=b+\delta(b)-\delta(a)$ and therefore $\delta(b)=\delta(a)$. This implies
$$(\beta)^{-1}*\alpha=\beta^{\beta(0)}*(\delta^{\alpha(0)})^{-1}=\delta^{\delta(b)-\delta(a)}=\delta^{0},$$

and $\alpha=\beta$. We can extend this arguments by limits to all $\Hull(\delta)$ to prove the injectivity of $g$.
Finally,  if
$G:[0,1]\times\Hull(\delta)\To\Hull(\delta)$ is defined  by
$$(c,\gamma)\goTo\gamma*\delta^{c\gamma(0)},$$
then $G$ is an isotopy from $g$ to the identity.
\end{proof}
The  next theorem follows from the argument above.
\begin{theorem}
If $f\in\homeo_{+}(\Ss)$ and  $\delta_{0}$ is the displacement function at the level 0, then $g:\Hull(\delta_{0})\To \Hull(\delta_{0})$  is semi-conjugated to $f$ by $K$.
\end{theorem}

\section{Induced homeomorphisms and its dynamics}\label{Inducedhomeos}
\subsection{Induced homeomorphisms}
 Given $f_{1}\in\Homeo_{+}(\UC),$ we can extend it to a homeomorphism $f\in\Homeo_{+}(\Ss)$ which will be our first example for the dynamics generated by a homeomorphism in $\Homeo_{+}(\Ss)$ and it´s relationship with the rotation set.\\

  Given $n\geq 2$, define $f_{n}:\R/n\Z\To\R/n\Z$ extending $f_{1}$ in the following way. Choose a lifting $F_{1}:\R\To\R$ which is a $\Z$-equivariant homeomorphism, and thus $n\Z$-equivariant,  and  project it to a homeomorphism  $f_{n}:\R/n\Z\to \R/n\Z.$ Note that these homeomorphisms satisfy the compatibility condition, i.e. if  $n|m$, then  the following  diagram commutes.

\[\xymatrix{\R/m\Z \ar[r]^{p_{nm}}\ar[d]^{f_{m}}& \R/n\Z\ar[d]^{f_{n}} \\
		\R/m\Z \ar[r]^{p_{nm}}& \R/n\Z}\]

 Therefore there is  a well defined  homeomorphism:

$$f:\Ss\To\Ss,$$

which covers $f_{1}$ in the sense that the following diagram commutes
\[\xymatrix{\Ss \ar[r]^{f}\ar[d]^{\pi_{1}}& \Ss\ar[d]^{\pi} \\
		\R/\Z \ar[r]^{f_{1}}& \R/\Z.}\]
We will call these kinds of homeomorphisms   \emph{induced homeomorphisms of degree 1}, and the  subspace of all these homeomorphisms will be denoted by $\Homeo_{I_{1}}(\Ss)$.\\

 We can write  each $f_{n}$  as $\id+\delta_{n},$ where $\delta_{n}:\R/n\Z\To\R/n\Z$ can be identified with a $n\Z$-invariant function $\delta_{n}:\R/n\Z\To\R$. In this case, each $\delta_{n}$ is  $\Z$-invariant  and bounded by 1, therefore $f$ can be written  as
  \begin{eqnarray*}
  f((x_{1},x_{2},\dots,x_{n},\dots))&=&(x_{1},x_{2},\dots,x_{n},\dots)\\
  &+&(\delta_{1}(x_{1}),\delta_{2}(x_{2}),\dots,\delta_{n}(x_{n}),\dots).
  \end{eqnarray*}
It is clear from the definitions that if $$f:=\id+\overline{\delta} \in\Homeo_{I_{1}}(\Ss)$$ is induced by  $$f_{1}:=\id+\delta_{1}\in\Homeo_{+}(\UC),$$ then the following diagram commutes.
\[\xymatrix{\R \ar@/^0.7cm/[rr]^{\delta_{0}}\ar[r]^{\sigma'} \ar[dr]^{\sigma} \ar[ddr]_{\pi}& \R\times\Zz\ar[d]^{P}\ar[r]^{\delta}& \R\\
	& \Ss\ar[ur]^{\overline{\delta}}\ar[d]^{\pi_{1}}&\\
	&\UC\ar[uur]_{\delta_{1}}&
}\]

%where $\sigma'(x)=(x,0).$ %as in the section \ref{limit}.\\

Note that the displacement function from $f$ is determined by
$$\overline{\delta}((x_1,x_2,\dots,x_{n},\dots))=(\delta_{1}(x_{1}),\delta_{2}(x_{2}),\dots,\delta_{n}(x_{n}),\dots).$$
It follows that $\overline{\delta}(x)\in\Ll_{0}$ for all $x\in\Ss,$ and we have proved the following theorem.

\begin{theorem}$
\Homeo_{I_{1}}(\Ss)\subset\Homeo_{+}(\Ss).$
\end{theorem}

\begin{remark}
  This homeomorphism is unique modulo integer translations in the base leaf due to the dependence of the choice of the lifting to $\R$ when we induce the homeomorphisms $f_{n},$ i.e. we have the  exact sequence

  \[\xymatrix{0\ar[r]&R_{\alpha}\ar[r]& \Homeo_{I_{1}}(\Ss)\ar[r]&\Homeo_{+}(\UC)\ar[r]& 1}\]
where $R_{\alpha}:=\{r_{\alpha}:\Ss\to\Ss|r_{\alpha}(s)=s+\alpha,\; \alpha\in i(\Z)\subset\Ss\}\cong\Z.$
\end{remark}
The universal central extension (see \cite{Ghys})
$$0\To\Z\To\HHomeo_{+}(\UC)\To\Homeo_{+}(\UC)\To 1$$
fits into the diagram
\[\xymatrix{0\ar[r]\ar[d]&\Z\ar[r]\ar[d]& \HHomeo_{+}(\UC)\ar[r]\ar[d]&\Homeo_{+}(\UC)\ar[r]\ar[d]& 1\ar[d]\\
	0\ar[r]&R_{\alpha}\ar[r]& \Homeo_{I_{1}}(\Ss)\ar[r]&\Homeo_{+}(\UC)\ar[r]& 1}\]
where $R_{\alpha}$ denotes integer translations in $\Ss$, the next isomorphism follows.

\begin{theorem}
  $\Homeo_{I_{1}}(\Ss)\simeq\HHomeo_{+}(\UC).$
\end{theorem}
%\begin{proof}
%Use the five lemma in the diagram
%\[\xymatrix{0\ar[r]\ar[d]&\Z\ar[r]\ar[d]& \HHomeo_{+}(\UC)\ar[r]\ar[d]&\Homeo_{+}(\UC)\ar[r]\ar[d]& 1\ar[d]\\
%	0\ar[r]&R_{\alpha}\ar[r]& \Homeo_{I_{1}}(\Ss)\ar[r]&\Homeo_{+}(\UC)\ar[r]& 1}\]
%\end{proof}

 It is important to notice that if $f\in\Homeo_{I_{1}}(\Ss)$ is induced by $f_{1}\in\Homeo_{+}(\Ss)$ then following diagram commutes.
 	\[\xymatrix{\Ss \ar[r]^{f}\ar[d]^{\pi_{1}}& \Ss\ar[d]^{\pi_{1}} \\
 		\R/\Z \ar[r]^{f_{1}}& \R/\Z}\]
In the next theorem we will see that this property  characterizes the induced homeomorphism of degree 1 in $\Homeo_{+}(\Ss)$.
\begin{theorem}
$f\in\Homeo_{I_{1}}(\Ss)$ if and only if $f\in\Homeo_{+}(\Ss)$ and the following diagram commutes.
 	\[\xymatrix{\Ss \ar[r]^{f}\ar[d]^{\pi_{1}}& \Ss\ar[d]^{\pi_{1}} \\
 		\R/\Z \ar[r]^{f_{1}}& \R/\Z.}\]
\end{theorem}
\begin{proof}
  ``$\Rightarrow$'' It follows directly from the last theorem and the definition.\newline
  ``$\Leftarrow$'' Any $f\in\Homeo_{+}(\Ss)$ preserve the leaves and the orientation and the following diagram commutes	
  \[\xymatrix{\Ss \ar[r]^{f}\ar[d]^{\pi_{1}}& \Ss\ar[d]^{\pi_{1}} \\
 		\R/\Z \ar[r]^{f_{1}}& \R/\Z,}\]

 so that $f$ is $\sigma(\Z)$--invariant. Therefore its lift will be as in the last theorem. We conclude that $f\in\Homeo_{I_{1}}(\Ss)$.
\end{proof}

Now we would like to generalize this idea to induced homeomorphisms of degree $n$. Specifically, we are thinking about elements in $\Homeo_{+}(\Ss)$ that satisfy, for some level $n\in\Z_{+}$, the following diagram
 	\[\xymatrix{\Ss \ar[r]^{f}\ar[d]^{\pi_{n}}& \Ss\ar[d]^{\pi_{n}} \\
 		\R/n\Z \ar[r]^{f_{n}}& \R/n\Z}\]
 		
We will give a description of the lifts of these kind of homeomorphisms, which we denote by $\Homeo_{I_{n}}(\Ss)$ for each $n\in\Z_{+}$, and we call these \emph{induced homeomorphism of degree $n$} by the homeomorphims $f_{n}\in\homeo_{+}(\R/n\Z)$.\\
\begin{theorem}\label{induced}
 $f\in\homeo_{I_{n}}(\Ss)$ is induced by a homeomorphisms $f_{n}\in\homeo_{+}(\UC)$ if and only if $f$ has a lift  $F:\R\times\Zz \to \R\times\Zz$ such that
$$(x,k)\goTo (F_{n}(x,k),k)=(F_{0}(x)+r(k),k)$$
where $F_{0}:\R\to\R$ is a lift of $f_{n}$ to $\R$ and $r(k)\in\{0,\dots,n-1\}$ is determined by $r(k)=i$ if  $k\in p_{n}^{-1}(i),$ and   $p_{n}$ denotes the canonical projection  $\Zz\To\Z/n\Z$.

\end{theorem}
\begin{proof}
Note that the restriction of $f$ to the base leaf is an increasing $n\Z$-invariant function. Therefore  $f|_{\Ll_{0}}:\Ll_{0}\To\Ll_{0}$ is a  $\sigma(n\Z)-$invariant function, and we must have that $$f(x+\sigma(n))=f(x)+\sigma(n).$$
 Let $F_{0}:\R\To\R$ be the restriction of a lift to $\R\times\Zz$  in level 0. Then $F_{0}$  can be seen as a $n\Z$-invariant lift to $\R$ of $f_{n}$,  and satisfies
  $$f(x+n)=f(x)+n.$$
  By the equivariance with respect to the diagonal action,  for each $k\in\Z,$
$$(F_{n}(x),k)=(F_{0}(x)+r(k),k)$$
where $r(k)=k\; (\mathrm{mod}\; n).$ Since the inclusion of $\Z$  in $\Zz$ is dense and the function

$$D:\Zz\To \Homeo_{+}(\R)$$
is continuous and constant on each integer in the open set $p_{n}^{-1}(i),$ it follows that $D$ is constant on the whole open set.
\end{proof}
\begin{corollary}
If two homeomorphisms are induced by the same homeomorphism $f_{n}\in\homeo_{+}(\R/n\Z),$ then they differ by an integer.
\end{corollary}

An important observation is that if $n|m$ for $n,m\in\Z$, then $\Homeo_{I_{n}}(\Ss)\subset\Homeo_{I_{m}}(\Ss)$ because  the following diagram commutes.

 \[\xymatrix{\Ss \ar[r]^{f}\ar[d]^{\pi_{n}}& \Ss\ar[d]^{\pi_{n}} \\
 		\R/n\Z \ar[r]^{f_{n}}\ar[d]^{p_{mn}}& \R/n\Z\ar[d]^{p_{mn}}\\
        \R/m\Z \ar[r]^{f_{m}}& \R/m\Z} \]

Here $f_{m}$ is the projection of $f_{n}$ by the covering function  $p_{mn}.$ Therefore we have an inductive system $\{\Homeo_{I_{n}}(\Ss)\}_{n\in\Z_{+}}$ with inclusion functions. Denote by $\Homeo_{I}(\Ss)$  the direct limit of this system,  then

$$\Homeo_{I}(\Ss)=\bigcup_{n\in\Z_{+}}\Homeo_{I_{n}}(\Ss).$$
\begin{theorem}
$\Homeo_{I}(\Ss)$ is an open and dense set in $\Homeo_{+}(\Ss)$
\end{theorem}
\begin{proof}
It is enough to prove that  $\Homeo_{I}(\Ss)$ is homeomorphic to $\Homeo_{+}^{p}(\Ss),$ which are the functions with periodic displacement. It is easy to see that if $f$ has periodic displacement  with period $p/q$, then $f$ is induced of degree $p$, so using Theorem \ref{induced}, we can finish the proof. It also follows that

$$\Homeo_{+}(\Ss)=\homeo_{I}(\Ss)\sqcup\Homeo_{+}^{\lp}(\Ss).$$
\end{proof}
\begin{remark}
$\Homeo_{I}(\Ss)$ has a subgroup structure with the induced operation.
\end{remark}

 	%\begin{remark}
 		%Before the next theorem we will do the basic  following observation:\\
 		
 		%If $\pi_{1}:\Ss\To\UC$ is the canonical projection, we have the inclusion
 		%$$\pi_{1}^{*}:\Char(\UC)\To\Char(\Ss)$$
 		
 		%And it's clear that $\pi_{1}^{*}(\id_{\UC})=\pi_1$. Is %importan to note that for the followings isomorphisms
 		
 		%\[\xymatrix{\Hom(\Char(\Ss),\R)\ar[r]&\Hom(\Q,\R)\ar[r]^{E_1}& \R&\Hom(\Z,\R)\ar[l]_{E_2}&\ar[l] \Hom(\Char(\UC),\R)}\]
 		
 		%the isomorphism $E_1$ and $E_2$ are determinated by $h_1\goTo h_1(1)$ and $h_2\goTo h_2(1),$ where $h_1\in\Hom(\Q,\R)$ and  $h_2\in\Hom(\Z,\R).$  Therefore the isomorphism between  $\Hom(\Char(\Ss),\R)$ and $\R,$ and the isomorphism between  $\Hom(\Char(\UC),\R)$ and $\R$ are determinated by:
 		
 		%\begin{center}$I_1\goTo I_1(\pi_1)$ and $I_2\goTo I_2(\id_{\UC})$.
 %	\end{center}
 		
 %	\end{remark}

\subsection{The dynamics of induced homeomorphisms}\label{dynamic}
First we will talk about the rational dynamics. It is obvious that in this case we do not have periodic points; so the next definition replaces the role of periodic points by periodic fibers.
\begin{definition}\label{fiberperiodic}
	Let $f\in\Homeo_{+}(\Ss)$. The point $s\in\Ss$ is called  \emph{$p/q$-fiber periodic} if there  exist $p,q\in\Z_+$ such that $$f^{q}(s)=s+\sigma(p),$$
	where  $\sigma:\R\To\Ss$ is the one-parameter subgroup. We will write just $s+p$ to denote $s+\sigma(p)$. Note that the name ``fiber periodic'' is due to the fact that the point is returning periodically  to the fiber in $s$ and we refer to the orbit of $s$ as a \emph{$p/q$-fiber}.
\end{definition}

The relationship between the dynamics generated by the homeomorphism  $f_{1}\in\Homeo_{+}(\UC)$ and the dynamics generated by the induced homeomorphism $f\in\Homeo_{+}(\Ss)$
is described as follows.

\begin{lemma}
	Let $f\in\Homeo_{+}(\Ss)$ be induced by a homeomorphism $f_{1}\in\Homeo_{+}(\UC)$. If $\rho(f)=p/q$ then $f$ has a $p/q$-fiber periodic point.
\end{lemma}
\begin{proof}
From Theorem \ref{induced} we know that the lift of $f$ to the covering space $\R\times\Zz$ is
$$F(x,k):=(F_{0}(x),k),$$
where $F_{0}:\R\to\R$ is a lift of $f$ to $\R$. Poincar\'e's theory says that for $F$ there exists  a point $x\in\R$ such that  $F(x)=x+1$.

The   set $$\bigg\{(F^{n}(x,k)|\; n\in\Z,\; k\in\Zz\bigg\}$$ is equivariant under the $\Z-$action and the projection to $\Ss$ satisfies the condition required.
\end{proof}

\begin{theorem}\label{dynamics}
Let $f\in\Homeo_{I}(\Ss)$ be induced by a homeomorphism $f_{1}\in\Homeo_{+}(\UC).$
\begin{enumerate}
\item If $\rho(f)=p/q,$ then any point $s\in\Ss$ is $p/q$-fiber periodic point or the orbit of $s$ is asymptotic to a $p/q$-fiber.
 \item If $\rho(f)\notin\Q$, then $f$ is semi-conjugate to the rotation by  $\rho(f).$
\end{enumerate}
\end{theorem}
\begin{proof}
 Using Theorem 4.6 of \cite{C-V}, it is sufficient to prove that  $f$ satisfies the bounded mean motion condition. It is easy to see that this condition holds since the restriction to any leaf is a  $\sigma(\Z)$-periodic function.
\end{proof}

\end{document}